\documentclass{amsart}
\usepackage[T1]{fontenc}
\usepackage[utf8]{inputenc}
\usepackage{amsmath,amsthm,amsfonts,amscd,amssymb,eucal,latexsym,mathrsfs}
\usepackage{stmaryrd}
\usepackage{enumerate}
\usepackage{hyperref}
\usepackage[all]{xy}
\usepackage{etoolbox}

\usepackage{tikz,tikz-cd}
\usetikzlibrary{shapes.geometric}
\usetikzlibrary{arrows}

\usetikzlibrary{positioning}
\usepackage{float}
\usepackage{MnSymbol}
\usetikzlibrary{matrix}

\newcommand{\toc}{\tableofcontents}

\theoremstyle{plain}
\newtheorem{theorem}{Theorem}[section]
\newtheorem*{theorem*}{Theorem}

\newtheorem*{corollary*}{Corollary}

\newtheorem{lemma}[theorem]{Lemma}

\theoremstyle{definition}
\newtheorem{remark}[theorem]{Remark}

\newtheorem{question}[theorem]{Question}

\newtheorem*{definition*}{Definition}

\usetikzlibrary{calc,graphs}
\usepackage{xcolor}
\usetikzlibrary{arrows,decorations.pathmorphing}

\newcommand{\p}{\varphi}

\newcommand{\e}{\varepsilon}

\renewcommand{\d}{\mathrm{d}}

\newcommand{\CI}{\mathbb{C}}

\DeclareMathOperator{\Prob}{\mathrm{Prob}}

\DeclareMathOperator{\ssi}{\Leftrightarrow}
\DeclareMathOperator{\impl}{\Rightarrow}

\newcommand{\IE}{\mathbb{E}}

\newcommand{\conv}{\operatorname{conv}}

\newcommand{\ts}{\textsection}

\newcommand{\ip}[1]{\langle#1\rangle} 
\newcommand{\norm}[1]{\|#1\|}

\newcommand{\St}{\mathop{\mathrm{St}}}

\newcommand{\hh}{\mathcal{H}}

\title{Positive definite maps on amenable groups, II} 

\author{Mika\"el Pichot}
\author{Erik S\'eguin}
\address{Mika\"el Pichot, McGill University, 805 Sherbrooke St W., Montr\'eal, QC H3A 0B9, Canada}\email{mikael.pichot@mcgill.ca}
\address{Erik S\'eguin, McGill University, 805 Sherbrooke St W., Montr\'eal, QC H3A 0B9, Canada}\email{erik.seguin@mail.mcgill.ca}

\begin{document}

\begin{abstract}
We introduce an Ulam-type stability condition for positive definite maps defined on a countable group and prove that this condition characterizes amenability. 
\end{abstract}

\maketitle

\section{Introduction}

This paper is a continuation of ref.\ \cite{stable} on Ulam stability \cite{Ulam} for positive definite maps on groups.

A countable discrete group $G$  is said to be \emph{Ulam stable} (following \cite{bot13,dc18,dcot19}) if for every ${\delta>0}$, there exists some ${\e>0}$ such that for every Hilbert space $\hh$ and every $\e$-representation ${\p\colon G\to U(\hh)}$, there exists a unitary representation ${\pi\colon G\to U(\hh)}$ such that 
\[\norm{\p(x)-\pi(x)}<\delta\]
for all ${x\in G}$, where $U(\hh)$ denotes the group of unitary operators acting on $\hh$,  $\|\cdot\|$ denotes the operator norm on the algebra $B(\hh)$ of bounded operators on $\hh$, and a map ${\p\colon G\to U(\hh)}$ is said to be an \emph{$\e$-representation} of $G$  if
\[
\|\p(xy)-\p(x)\p(y)\|\leq \e
\] 
for all $x,y\in G$.

It is a well--known theorem of Kazhdan \cite[Theorem 1]{kazhdan82} that every countable discrete amenable group  \cite{Greenleaf} is Ulam stable. The converse is an open problem (see in particular  \cite[Question 1.5]{dcot19}).

Ulam stability can also be formulated in terms of positive definite maps. 
Namely,  the following two assertions are equivalent (see \cite[Prop.\ 1.2]{stable}) for every group $G$ and Hilbert space $\hh$:
\begin{enumerate}
\item[(0)]  there exist real numbers $\kappa>0$ and $0<\delta<1$ such that for every ${0<\e<\delta}$ and every unitary $\e$-representation ${\p\colon G\to U(\hh)}$  there exists a unitary representation ${\pi\colon G\to U(\hh)}$ such that  
\[
\|\p(x)-\pi(x)\|\leq\kappa\e
\]
for all $x\in G$;
\item[(0')] there exist real numbers $\kappa_1,\kappa_2>0$, $0<\delta<1$, and $p>1$ such that for every $0<\e<1$ such that $\kappa_1\e^{p-1}\leq\delta^{p-1}$   and every $\e$-representation ${\p\colon G\to U(\hh)}$  there exists a unital positive definite ${\kappa_1\e^p}$-representation ${\psi:G\to B(\hh)}$ such that 
\[
\norm{\p(x)-\psi(x)}\leq\kappa_2\e
\]
for all ${x\in G}$.
\end{enumerate}

\noindent  In this equivalence, the first assertion is a variation on  Ulam stability,  in which $\e$ is required to depend linearly on $\delta$, and the second assertion can be understood as an Ulam stability condition for sufficiently multiplicative positive definite maps on the group $G$.

In the present paper we propose a new stability condition for positive definite maps on $G$ and prove that it characterizes amenability.   This extends the results of \cite{stable} and  provides further justification for the shift in emphasis from unitary representations to positive definite maps. 


Let us first quote results from \cite{stable} to motivate the new condition. 
A map $\p\colon G\to B(\hh)$ is said to be \emph{uniformly bounded} if 
\[
{\|{\p}\|:=\sup_{x\in G}\,\|{\p(x)}\|<\infty}.
\]   
The following characterization of amenability is part of Theorem 1.3 of \cite{stable}.

\begin{theorem}\label{T - stable main}
The following are equivalent: 
\begin{enumerate}
\item $G$ is amenable;
 \item there exists a state ${\tau\in\mathrm{St}(\ell^\infty(G))}$ such that for  every uniformly bounded map ${\p:G\to B(\hh)}$, there exists a positive definite map ${\psi:G\to B(\hh)}$ such that 
\[
{\Phi(\p(x)^*\,\psi(x))=\tau_y\,\Phi(\p(x)^*\,\p(xy)\,\p(y)^*)}
\]
for all ${x\in G}$ and every normal linear functional ${\Phi\in B(\hh)_*}$.
\end{enumerate}
\end{theorem}

 The discussion of this result in \cite{stable} isolates two important characteristics of Condition (2):

\begin{enumerate}
\item[(A)] the fact that the state $\tau$ in  (2) provides an invariant mean on a group $G$; in fact,  Condition (2) can be viewed as a characterization of the invariance of a mean  in terms of positive definite maps;
\item[(B)]  the fact that (2) applies to all uniformly bounded maps, as opposed to restricting to  $\e$-representations.
\end{enumerate}

\noindent We look for an  Ulam-type  stability condition for positive definite maps which takes these directions into account. By Problem (A) and Problem (B), we mean the formulation of a condition characterizing amenability analogous to (2) which  does not refer to the state $\tau$, and,  respectively,  restricts to $\e$-representations.

Considering the equality 
\[
{\Phi(\p(x)^*\,\psi(x))=\tau_y\,\Phi(\p(x)^*\,\p(xy)\,\p(y)^*)},\ \  \forall {x\in G},\ \  \forall{\Phi\in B(\hh)_*}
\]
in relation with Problem (B), the following result was established in \cite{stable}.  

\begin{theorem} 
Let $G$ be a group and $\hh$ be a Hilbert space.
Suppose that there exists a state ${\tau\in\mathrm{St}(\ell^\infty(G))}$ and $0<\delta<1$ such that for every   $\delta$-representation ${\p:G\to U(\hh)}$, there exists a positive definite map ${\psi:G\to B(\hh)}$ such that 
\[
{\Phi(\p(x)^*\,\psi(x))=\tau_y\,\Phi(\p(x)^*\,\p(xy)\,\p(y)^*)}
\]
for all ${x\in G}$ and every normal linear functional ${\Phi\in B(\hh)_*}$. 

Then $G$ is Ulam stable.
\end{theorem}

Here is a tentative formulation of an Ulam-type stability condition for positive definite maps:
\begin{enumerate}
\item[(3)] for every ${\e>0}$, every Hilbert space $\hh$, and every unitary $\e$-representation ${\varphi:G\to U(\hh)}$, there exists a positive definite map ${\psi:G\to B(\hh)}$ such that 
\[
\norm{\p(x)-\psi(x)}\leq\e
\]
for all $x\in G$.
\end{enumerate}

Note that Condition (3) does not refer to a state on $\ell^\infty(G)$ and the statement applies to $\e$-representations ${\varphi:G\to B(\hh)}$ for all $\e>0$.

\begin{question} Does there exist a non-amenable countable discrete group $G$ for which Condition (3) holds?
\end{question}

Thus, a negative answer to this question would resolve both Problem (A)  and Problem (B).

 Condition (3) should be compared with the conclusion of Kazhdan's theorem in the case of unitary representations, by which we mean the following assertion (which corresponds to Condition (0) and Condition (4) in \cite[Prop.\ 1.2]{stable} at $\kappa=2$):
\begin{itemize} 
\item[]  there exists a real number $0<\delta<1$ such that for every ${0<\e<\delta}$, every Hilbert space $\hh$, and every unitary $\e$-representation ${\p\colon G\to U(\hh)}$,  there exists a unitary representation ${\pi\colon G\to U(\hh)}$ such that  
\[
\|\p(x)-\pi(x)\|\leq2\e
\]
for all $x\in G$
\end{itemize}

We shall make 
some remarks  in particular on  the ``numerical'' discrepancies between the above assertion for unitary maps and Condition (3) for positive definitive maps.

The first remark concerns the constant $\kappa=2$ (in the notation of \cite[Prop.\ 1.2(4)]{stable}) in the conclusion of Kazhdan's theorem. This constant is not intrinsic to the group $G$ and it can be removed   by using positive definite maps  (namely, the fact that  $G$ is amenable implies (3) with constant 1).

The second remark concerns the constant $\delta$ (with respect to $\kappa=2$). Kazhdan proves that $\delta=\frac 1 {200}$ is small enough. The corresponding ``Ulam stability gap'', namely, the interval $[\delta,1]$, can be reduced, for example by using the techniques developed in \cite{dc18,dcot19}. Thus, one may choose $\delta = \frac 1 4$ for instance (see \cite{msct}). (Since $\kappa=2$, stability is tautological for $\e>1$.) This constant is also unnecessary in the case of positive definite maps (namely, the fact that  $G$ is amenable  implies (3) without imposing a restriction on the values of $\e$).

One may introduce several variations on Condition (3). 
For example, the following formulation was put forward  in \cite[Question 8.3]{stable}:
\begin{enumerate}
\item[(4)] for every Hilbert space $\hh$ and every unitary map ${\varphi:G\to U(\hh)}$, there exists a positive definite map ${\psi:G\to B(\hh)}$ such that 
\[
{\norm{\varphi(x)-\psi(x)}\leq\sup_{y\in G} \norm{\varphi(x)\varphi(y)-\varphi(xy)}}
\]
for all $x\in G$.
\end{enumerate}

Question 8.3 in \cite{stable}, which is open, asks if this condition characterizes amenability.
An equivalence between the amenability of $G$ and Condition (4) would resolve Problem (A), and would partially resolve  Problem (B) in the sense that  (4) ensures that $\psi$ is close to $\p$ in the case that 
$\p$ is an $\e$-representation.

In the present paper, we consider a variation  of Condition (4) in which the ``proximity'' relation between $\psi$ and $\p$ is expressed in the strong operator topology, rather than  the operator norm. Furthermore, we prove that the new formulation is equivalent to amenability.

This  resolves Problem (A), because Condition (5) does not make reference to a state on $\ell^\infty(G)$, and this resolves Problem (B) partially, in the sense that Condition (5) ensures that $\psi$ is close to $\p$---in the operator norm---in the case that 
$\p$ is an $\e$-representation. 

Our main result can be stated as follows.

\begin{theorem}\label{T - main theorem}
Let $G$ be a countable discrete group. 
 The following are equivalent:
 \begin{enumerate}
 \item $G$ is amenable;
 \item[(5)] for every (separable) Hilbert space $\hh$ and  every uniformly bounded map ${\varphi:G\to B(\hh)}$,  there exists a positive definite map ${\psi:G\to B(\hh)}$ such that for every finite set $F\subset G$, every integer  $n\geq 1$, and every finite family of functions  $\xi_1,\ldots, \xi_n,\zeta_1,\ldots \zeta_n\colon F\to \hh$, there exists  a $y\in G$ such that 
\begin{align*}
\sum_{i=1 }^n\sum_{x\in F }\norm{\varphi(xy)\,\varphi(y)^*\,\xi_i(x)-\psi(x)\,\xi_i(x)}^2&\\
&\hskip-4cm \leq\sum_{i=1 }^n\sum_{x\in F }\norm{\varphi(xy)\,\varphi(y)^*\,\xi_i(x)-\zeta_i(x)}^2.
\end{align*}
\end{enumerate}
\end{theorem}

The rest of the paper is devoted to the proof of Theorem \ref{T - main theorem}.  
It relies on the results of \cite{stable},  Tarski's characterization of amenability (in terms of paradoxical decompositions \cite{Wagon}), and uses a new characterization of closed convex hulls for a set of bounded operators with respect to the strong operator topology (see \ts\ref{ S - convex}). 

We also provide an alternative argument relying on compactness rather than Tarski's theorem (see Remark \ref{R - Tarski remark}).

For the sake of completeness, let us prove here the remark we made before the statement of the main theorem: 

\begin{remark} Condition (5) implies that $\psi$ is close to $\p$ in the operator norm assuming that $\p$ is a unitary $\e$-representation.
\end{remark}

\begin{proof} Choose
\begin{enumerate}
\item[] $F=\{x\}$
\item[] $n=1$
\item[] $\zeta(x)=\p(x)\xi(x)$
\end{enumerate}
for all $x\in G$ and all $\xi\colon F\to \hh$. This implies
\[
\|\p(x)-\psi(x)\|\leq 2\e,
\]
for every $x\in G$, by taking the supremum over $\xi\colon F\to \hh$ of norm 1.
Note that $\kappa=2$ in this case.
\end{proof}

 \bigskip

\noindent\textbf{Acknowledgment.}   The authors are partially funded by NSERC Discovery Fund 234313. 

\toc

\section{SOT closed convex hulls}\label{ S - convex}
We shall begin with a new characterization of closed convex hulls in $B(\hh)$ with respect to  the strong operator topology:

\begin{theorem}\label{ T - SOT hull}
Let $\hh$ be a Hilbert space, $T\in B(\hh)$ be a bounded operator, $X$ be a nonempty set,  and ${\p:X\to B(\hh)}$ be a map. The following are equivalent:
\begin{enumerate}
\item $\forall \xi_1,\ldots, \xi_n\in \hh$, $\forall \eta_1,\ldots, \eta_n\in \hh$,  $\exists x\in X$ such that
\[
\sum_{i=1}^n\norm{\varphi(x)\,\xi_i-T\xi_i}^2\leq\sum_{i=1}^n\norm{\varphi(x)\,\xi_i-\eta_i}^2;
\]
\item  ${T\in\overline{\mathrm{conv}\{\varphi(x):x\in X\}}}$, where the closure is with respect to the strong operator topology.
\end{enumerate}
\end{theorem}

We require the following lemma which characterizes closed convex hulls in Hilbert spaces in terms  of closed half-spaces associated with certain \emph{orthogonal} (namely, median) hyperplanes.  It could be viewed as a sort of ``separation theorem'', although it is of course  restricted in scope compared to the Hahn--Banach theorem due to its Hilbertian nature.

\begin{lemma}\label{L - hyperplanes} Let $\hh$ be a Hilbert space, $\xi\in \hh$ be a vector,  $X$ be a nonempty set, and ${f:X\to\hh}$ be a map. The following are equivalent 
\begin{enumerate}
\item $\forall\eta\in\hh$, $\exists x\in X$ such that
\[
\norm{f(x)-\xi}\leq\norm{f(x)-\eta};
\]
\item ${\xi\in\overline{\mathrm{conv}\{f(x):x\in X\}}}$, where the closure is with respect to the norm topology. 

\end{enumerate}
\end{lemma}

\begin{proof}
Let ${C=\overline{\mathrm{conv}\{f(x):x\in X\}}}$ and suppose that 
\[\forall\eta\in\hh,\ \exists x\in X:\norm{f(x)-\xi}\leq\norm{f(x)-\eta}\]
and ${\xi\notin C}$. Let $\eta$ denote the orthogonal projection of $\xi$ onto $C$. Clearly, $\eta\neq \xi$. The median hyperplane for $[\xi,\eta]$ divides $\hh$ into two disjoint open half-spaces; we denote by $H$ the open half-space which contains $\eta$. Since $\eta$ is the orthogonal projection of $\xi$ onto $C$, the closed convex set $C$ is included in $H$, and therefore
\[
\norm{\zeta-\eta}<\norm{\zeta-\xi}
\]
 for every $\zeta\in C$.
For ${x\in X}$ as above we obtain 
\[
\norm{f(x)-\xi}\leq\norm{f(x)-\eta}< \norm{f(x)-\xi}
\]
which is a contradiction.

Conversely,   suppose that 
\[
\exists\eta\in\hh, \forall x\in X:\norm{f(x)-\xi}>\norm{f(x)-\eta}
\]
and ${\xi\in C}$. Note that  $\xi\neq \eta$. Let $H$ denote the closed half-space of the median hyperplane for $[\xi,\eta]$ which contains $\eta$. Clearly, $f(x)\in H$ for every $x\in X$. Since $H$ is closed and convex, this implies that $C\subset  H$. 

Thus, $\xi\in  H$, which is a contradiction.
\end{proof}

\begin{proof}[Proof of Theorem \ref{ T - SOT hull}] 
Consider $T\in B(\hh)$ such that $\forall \xi_1,\ldots, \xi_n\in \hh$, $\forall \eta_1,\ldots, \eta_n\in \hh$,  $\exists x\in X$ such that
\[
\sum_{i=1}^n\norm{\varphi(x)\,\xi_i-T\xi_i}^2\leq\sum_{i=1}^n\norm{\varphi(x)\,\xi_i-\eta_i}^2.
\]
Fix $\xi_1,\ldots, \xi_n\in \hh$.  Consider $\xi =  (T\xi_1,\ldots, T\xi_n)\in\hh^{\oplus n}$ and let $\eta =  (\eta_1,\ldots, \eta_n)\in\hh^{\oplus n}$ be an arbitrary vector. By assumption, there exists $x\in X$ such that
\begin{align*}
\norm{f(x)-\xi}^2&= \sum_{i=1}^n\norm{\varphi(x)\,\xi_i-T\xi_i}^2 \\
&\leq \sum_{i=1}^n\norm{\varphi(x)\,\xi_i-\eta_i}^2 \\
&=\norm{f(x)-\eta}^2.
\end{align*}
By Lemma \ref{L - hyperplanes}, ${\xi\in\overline{\mathrm{conv}\{f(x):x\in X\}}}$,  and therefore there exists a net $(\mu_\alpha)$ of finitely supported probability measures on $X$ such that
\[
\int_X f\d\mu_\alpha \to \xi
\]
in the norm topology. Let 
\[
T_\alpha := \int_X \p\d\mu_\alpha\in \conv\{\p(x) : x\in X\}.
\]
Then
\[
\sum_{i=1}^n\|T_\alpha\xi_i-T\xi_i\|^2=\big\|\int_X f_\alpha\d\mu_\alpha - \xi\big\|^2\to 0.
\]
This shows that $T\in \overline{\mathrm{conv}\{\varphi(x):x\in X\}}$, where the closure is with respect to the strong operator topology.

Conversely,  suppose $T\in \overline{\mathrm{conv}\{\varphi(x):x\in X\}}$,  where the closure is with respect to the strong operator topology, and let
$\xi_1,\ldots, \xi_n\in \hh$, $ \eta_1,\ldots, \eta_n\in \hh$.  Define $f\colon X\to \hh^{\oplus n}$, $\xi,\eta\in\hh^{\oplus n}$ as above. Then  
\[
{\xi\in\overline{\mathrm{conv}\{f(x):x\in X\}}},
\] 
where the closure is with respect to the norm topology, and therefore, by Lemma \ref{L - hyperplanes}, there exists an $x\in X$ such that
\[
\norm{f(x)-\xi}\leq\norm{f(x)-\eta}.
\]
Therefore,
\[
\sum_{i=1}^n\norm{\varphi(x)\,\xi_i-T\xi_i}^2\leq\sum_{i=1}^n\norm{\varphi(x)\,\xi_i-\eta_i}^2
\]
establishing (1).
\end{proof}

\begin{remark}\label{R - SOT}
The above result may also be compared to the von Neumann bicommutant theorem, which characterizes the SOT closures of unital $*$-algebras; furthermore, it can be used (and was discovered) to answer certain questions in the majorization  theory, in particular regarding the self-adjointness assumption on unitary orbits. 
\end{remark} 

\section{Positive definite maps and Tarski's theorem}\label{S - Tarski}

The following characterization of amenability follows by combining the techniques developed in \cite{stable} and Tarski's theorem.

\begin{theorem}\label{T - Tarski positive}
Let $G$ be a countable discrete group. The following are equivalent:
\begin{enumerate}
\item $G$ is amenable
\item[(6)] for every separable Hilbert space $\hh$ and every uniformly bounded map $\p\colon G\to B(\hh)$, there exists a net $(\mu_\alpha)_\alpha\subset \Prob(G)$ and a positive definite map $\psi\colon G\to B(\hh)$ such that
\[
\lim_{\alpha} \int_G \p(xy)\p(y)^*\d\mu_\alpha(y)= \psi(x)\ \ \ \ \forall x\in G
\]
where the limit is with respect to the strong operator topology.
\end{enumerate}
\end{theorem}

We subdivide the proof into three equivalences (6) $\ssi$ (7) $\ssi$  (8) $\ssi$ (1),  where  (7) and (8) are stated below. Since the group $G$ is countable, one may replace the net $(\mu_\alpha)$ by a sequence in these statements.  

\begin{lemma}
Let $G$ be a countable discrete group. The following are equivalent: 
\begin{enumerate}
\item[(6)] for every separable Hilbert space $\hh$ and  every uniformly bounded map $\p\colon G\to B(\hh)$, there exists a net $(\mu_\alpha)_\alpha\subset \Prob(G)$ and a positive definite map $\psi\colon G\to B(\hh)$ such that
\[
\lim_{\alpha} \int_G \p(xy)\p(y)^*\d\mu_\alpha(y)= \psi(x)\ \ \ \ \forall x\in G
\]
where the limit is with respect to the strong operator topology.
\item[(7)] for every separable Hilbert space $\hh$ and every finite set of uniformly bounded maps ${\varphi_1,\dots,\varphi_n\in\ell^\infty(G,B(\hh))}$, there exists a net ${(\mu_\alpha)}_\alpha\subset \Prob(G)$ and positive definite maps ${\psi_1,\dots,\psi_n\in\ell^\infty(G,B(\hh))}$ such that 
\[
\lim_{\alpha} \int_G \p_i(xy)\p_i(y)^*\d\mu_\alpha(y)= \psi_i(x)\ \ \ \ \forall x\in G\text { and }i=1,\ldots,n
\]
where the limit is with respect to the strong operator topology.
\end{enumerate}
\end{lemma}

\begin{proof}
Suppose that ${(6)}$ holds.  Fix a separable Hilbert space $\hh$ and uniformly bounded maps ${\varphi_1,\dots,\varphi_n\in\ell^\infty(G,B(\hh))}$. For ${i=1,\dots,n}$, let ${U_i:\hh\hookrightarrow\hh^{\oplus n}}$ be the canonical inclusion map onto the $i$th coordinate and let ${\varphi:G\to B(\hh^{\oplus n})}$ be the uniformly bounded map defined by 
\[
\varphi(x)=U_1\,\varphi_1(x)\,U_1^*+\cdots+U_n\,\varphi_n(x)\,U_n^*.
\]
Applying ${(1)}$ to $\varphi$ then yields a net ${(\mu_\alpha)}$ in $\Prob(G)$ and a positive definite map ${\psi:G\to B(\hh^{\oplus n})}$ such that 
\[
\lim_{\alpha} \int_G \p(xy)\p(y)^*\d\mu_\alpha(y)= \psi(x)\ \ \ \ \forall x\in G
\]
where the limit is with respect to the strong operator topology.
For ${i=1,\dots,n}$, let ${\psi_i:G\to B(\hh)}$ be the positive definite map defined by 
\[
\psi_i(x)=U_i^*\,\psi(x)\,U_i
\]
As ${\varphi_i(x)=U_i^*\,\varphi(x)\,U_i}$ for ${i=1,\dots,n}$ and all ${x\in G}$, it follows that 
\begin{align*}
\big\|\int_G \p_i(xy)\p_i(y)^*\xi\d\mu_\alpha(y)&-\psi_i(x)\,\xi\big\|\\
&=\big\|\int_G U_i^*\p(xy)U_iU_i^*\p(y)^*U_i\xi\d\mu_\alpha(y)-U_i^*\psi(x)U_i\,\xi\big\|\\
&\leq\norm{U_j^*}\,\big\|\int_G \p(xy)\p(y)^*U_i\xi\d\mu_\alpha(y)-\psi(x)U_i\,\xi\big\|\to0
\end{align*}
for all ${x\in G}$ and ${\xi\in\hh}$, which proves the implication ${(6)\Rightarrow(7)}$. The converse is clear.
\end{proof}

\begin{lemma}
Let $G$ be a countable discrete group. The following are equivalent: 
\begin{enumerate}
\item[(7)] for every separable Hilbert space $\hh$ and every finite set of uniformly bounded maps ${\varphi_1,\dots,\varphi_n\in\ell^\infty(G,B(\hh))}$, there exists a net ${(\mu_\alpha)}_\alpha\subset \Prob(G)$ and positive definite maps ${\psi_1,\dots,\psi_n\in\ell^\infty(G,B(\hh))}$ such that 
\[
\lim_{\alpha} \int_G \p_i(xy)\p_i(y)^*\d\mu_\alpha(y)= \psi_i(x)\ \ \ \ \forall x\in G\text { and }i=1,\ldots, n
\]
where the limit is with respect to the strong operator topology.
\item[(8)] for every separable  Hilbert space $\hh$ and every finite set of uniformly bounded maps ${\varphi_1,\dots,\varphi_n\in\ell^\infty(G,B(\hh))}$, there exists a conditional expectation 
\[
\IE\colon \ell^\infty(G,B(\hh)) \to B(\hh)
\] 
such that the maps $\psi_i\colon G\to B(\hh)$ defined by
\[
 \psi_i(x) = \IE_y  \p_i(xy)\p_i(y)^*
\]
are positive definite for all $i=1,\ldots,n$.
\end{enumerate}
\end{lemma}

\begin{proof}
The forward implication follows by taking a weak$^*$ cluster point of the net $(\mu_\alpha)_\alpha$ in the state space $\St(\ell^\infty(G))$ and considering the associated conditional expectation $\IE\colon \ell^\infty(G,B(\hh)) \to B(\hh)$.  The converse implication follows by using the density of $\Prob(G)$ in $\St(\ell^\infty(G))$ with respect to the weak$^*$ topology. 
\end{proof}

We shall deduce the last lemma from Tarski's theorem.

\begin{lemma}
Let $G$ be a countable discrete group. The following are equivalent: 
\begin{enumerate}
\item[(1)] $G$ is amenable;
\item[(8)] for every separable Hilbert space $\hh$ and every finite set of uniformly bounded maps ${\varphi_1,\dots,\varphi_n\in\ell^\infty(G,B(\hh))}$, there exists a conditional expectation 
\[
\IE\colon \ell^\infty(G,B(\hh)) \to B(\hh)
\] 
such that the maps $\psi_i\colon G\to B(\hh)$ defined by
\[
 \psi_i(x) = \IE_y  \p_i(xy)\p_i(y)^*
\]
are positive definite for all $i=1,\ldots,n$.
\end{enumerate}
\end{lemma}

\begin{proof}
If $G$ is amenable, every invariant mean on $G$ extends to a conditional expectation $\IE\colon \ell^\infty(G,B(\hh)) \to B(\hh)$  such that for every uniformly bounded map 
$\p\colon G\to B(\hh)$, the map $\psi\colon G\to B(\hh)$ defined by
\[
 \psi(x) = \IE_y  \p(xy)\p(y)^*, \ \ \forall x\in G
\]
is positive definite (see \cite{dcot19}). This is a strengthening of Condition (8)  in which the conditional expectation is uniform. That this statement actually characterizes amenability was observed in \cite{stable}.

Conversely, suppose that (8) holds for $\hh=\CI$ and $G$ is not amenable.  Then there exist elements ${s_1,\dots,s_n,t_1,\dots,t_m\in G}$ and pairwise disjoint subsets 
\[
{A_1,\dots,A_n,B_1,\dots,B_m\subset G}
\] 
such that 
\[
G=\bigcup_{j\,=\,1}^n s_j\mskip0.5\thinmuskip A_j=\bigcup_{k\,=\,1}^m t_k\mskip0.5\thinmuskip B_k
\]
Let $X$ denote the set of elements of the form $i^l\chi_G$, $i^l\chi_{A_j}$, $i^l\chi_{B_k}$ for every $1\leq j\leq n,\,1\leq k\leq m, 0\leq l <4$. 
Let $\ip{\ ,\ } \colon \ell^\infty(G)\times\ell^\infty(G)\to\CI$ be the positive semi-definite sesquilinear form defined by 
\[
\ip{f, g}=\IE(f^*g).
\]
As $\psi_f$ is positive definite for all $f\in X$, in turn ${\ip{\lambda(\,\cdot\,)\,f,\ f}}$ is positive definite for all ${f\in X}$; as $X$ is stable under multiplication by $i$, it follows by \cite[Lemma 6.1]{stable} that 
\[
\IE(\lambda(s)\,\chi_{A_j})=\ip{\chi_G,\lambda(s)\,\chi_{A_j}}=\ip{\lambda(s^{-1})\,\chi_G,\ \chi_{A_j}}=\ip{\chi_G,\ \chi_{A_j}}=\IE(\chi_{A_j})
\]
for all ${s\in G}$ and ${j=1,\dots,n}$, and similarly 
\[
{\IE(\lambda(s)\,\chi_{B_k})=\IE(\chi_{B_k})}
\]
for all ${s\in G}$ and ${k=1,\dots,m}$. However,  this implies that
\begin{align*}
\IE(\chi_G)+\IE(\chi_G)&=\IE(\lambda(s_1)\,\chi_{A_1})+\cdots+\IE(\lambda(s_n)\,\chi_{A_n})\\
&\ \ \ \ \ \ \ \ \ +\IE(\lambda(t_1)\,\chi_{B_1})+\cdots+\IE(\lambda(t_m)\,\chi_{B_m})
\\&=\IE(\chi_{A_1})+\cdots+\IE(\chi_{A_n})+\IE(\chi_{B_1})+\cdots+\IE(\chi_{B_m})
\\&=\IE(\chi_{A_1}+\cdots+\chi_{A_n}+\chi_{B_1}+\cdots+\chi_{B_m})
\\&\leq\IE(\chi_G)
\end{align*}
which is a contradiction, since ${\IE(\chi_G)=1}$. It thus follows that $G$ is amenable. 
\end{proof}

\begin{remark}\label{R - Tarski remark}
Instead of Tarski's theorem in (8) $\impl$ (1), one can use a compactness argument to construct  an invariant mean.  Namely, suppose that (8) holds for $\hh=\CI$, so that for  every finite set of uniformly bounded maps ${\varphi_1,\dots,\varphi_n\in\ell^\infty(G,B(\hh))}$, there exists a state $\tau^{\varphi_1,\dots,\varphi_n}$ such that the maps $\psi_i\colon G\to B(\hh)$ defined by
\[
 \psi_i(x) = \tau_y^{\varphi_1,\dots,\varphi_n} \p_i(xy)\p_i(y)^*
\]
are positive definite for all $i=1,\ldots,n$. Since $\St(\ell^\infty (G))$ is compact in the weak$^*$ topology,  the net 
\[
\{\tau^{\varphi_1,\dots,\varphi_n} : \varphi_1,\dots,\varphi_n\in \ell^\infty(G,B(\hh))\}
\] 
obtained by ordering the finite subsets of $\ell^\infty(G,B(\hh)$ by inclusion admits a cluster point $\tau$ such that
\[
 \psi(x) = \tau_y\p(xy)\p(y)^*
\]
is positive definite for every $\p\in \ell^\infty(G)$.
It follows by  \cite[Corollary 6.5]{stable} that $G$ is amenable.
\end{remark}

\section{Proof of Theorem \ref{T - main theorem}}\label{S - Proof}

Let $\hh$ be a separable Hilbert space and  ${\varphi:G\to B(\hh)}$ be a uniformly bounded map.
We assume  that Condition (5) holds:
\begin{quote}
there exists a positive definite map ${\psi:G\to B(\hh)}$ such that for every finite set $F\subset G$, every $n\geq 1$, and every finite family of functions $\xi_1,\ldots, \xi_n,\zeta_1,\ldots \zeta_n\colon F\to \hh$, there exists  a $y\in G$ such that 
\begin{align*}
\sum_{i=1 }^n\sum_{x\in F }\norm{\varphi(xy)\,\varphi(y)^*\,\xi_i(x)-\psi(x)\,\xi_i(x)}^2&\\
&\hskip-4cm \leq\sum_{i=1 }^n\sum_{x\in F }\norm{\varphi(xy)\,\varphi(y)^*\,\xi_i(x)-\zeta_i(x)}^2.
\end{align*}
\end{quote}
\noindent and establish Condition (6). 

Let ${\psi:G\to B(\hh)}$ be positive definite satisfying (5). Let $F\subset G$ be a finite set.  
 Since $\p$ is uniformly bounded, the map 
\begin{align*}
\p_F\colon  G&\to\ \ \ \ \ \ \ \ B(\bigoplus_{x\in F} \hh)\\
y&\mapsto  [\xi\mapsto (\p(xy)\p(y^*)\xi(x)))_{x\in F}]
\end{align*}
is uniformly bounded.

Let $T\in B(\bigoplus_{x\in F} \hh)$ be the operator defined by
\[
T(\xi)=(\psi(x)\xi(x)))_{x\in F}.
\]
Let $\xi_1,\ldots, \xi_n,\zeta_1,\ldots, \zeta_n\colon X\to \hh$. By assumption,  there exists a $y\in G$ such that
\[
\sum_{i=1 }^n\sum_{x\in F }\norm{\varphi(xy)\,\varphi(y)^*\,\xi_i(x)-\psi(x)\,\xi_i(x)}^2\leq\sum_{i=1 }^n\sum_{x\in F }\norm{\varphi(xy)\,\varphi(y)^*\,\xi_i(x)-\zeta_i(x)}^2
\]
Thus,
\[
\sum_{i=1 }^n\|\p_F(y)\xi_i-T\xi_i\|^2\leq\sum_{i=1 }^n\|\p_F(y)\xi_i-\zeta_i\|^2. 
\]
By Theorem  \ref{ T - SOT hull}, $T\in\overline {\conv\{\p_F(y): y\in G\}}$. Therefore there exists a net $(\mu_\alpha)_\alpha\subset \Prob(G)$ such that 
\[
\lim_\alpha \int_G \p(xy)\p(y)^*\d\mu_\alpha(y) = \psi(x)\ \ \forall x\in F
\]
where the limit is with respect to the strong operator topology. It follows by a diagonal argument that Condition (6) holds. The converse follows by reversing the steps and therefore (5) is equivalent to (6). 

By Theorem \ref{T - Tarski positive}, (6) characterizes the amenability of $G$.
This concludes the proof of Theorem \ref{T - main theorem}.

\end{document}